\documentclass[11pt]{amsart}
\usepackage[top=1.2in,bottom=1.2in,left=1.2in,right=1.2in]{geometry}
\usepackage[inline]{enumitem}
\usepackage{amsthm}

\usepackage{mathtools}
\DeclarePairedDelimiter\floor{\lfloor}{\rfloor}
\usepackage{float}
\usepackage{amssymb}
\usepackage{xypic}
\usepackage{hyperref}
\hypersetup{
    colorlinks,
    citecolor=black,
    filecolor=black,
    linkcolor=black,
    urlcolor=black
}

\newcommand{\Z}{\mathbb{Z}}

\newcommand{\B}{B}

\DeclareMathOperator{\Sp}{Sp}

\DeclareMathOperator{\Mod}{Mod}

\newcommand{\p}[1]{\bigskip \noindent \emph{#1}.}

\theoremstyle{plain}

\newtheorem{theorem}{Theorem}[section]
\newtheorem{proposition}[theorem]{Proposition}

\newtheorem{lemma}[theorem]{Lemma}

\newtheorem{question}{Question}

\newcommand{\nc}{\newcommand}
\nc{\dmo}{\DeclareMathOperator}
\nc{\para}[1]{\medskip\noindent\textbf{#1.}}

\title{Finite quotients of braid groups}

\author{Alice Chudnovsky}

\author{Kevin Kordek}

\author{Qiao Li}

\author{Caleb Partin}

\address{Alice Chudnovsky \\ Department of Mathematics\\ University of Illinois at Urbana-Champaign \\ 1409 W. Green Street (MC-382) \\ Urbana, IL 61801}
\email{achudnovsky98@gmail.com}

\address{Kevin Kordek \\ School of Mathematics\\ Georgia Institute of Technology \\ 686 Cherry St. \\ Atlanta, GA 30332}
\email{kevin.kordek@math.gatech.edu}

\address{Qiao Li \\ Department of Mathematics\\ University of California, Berkeley \\ 970 Evans Hall $\#$ 3840 \\ Berkeley, CA 94720-3840 }
\email{lilyli@berkeley.edu}

\address{Caleb Partin \\ School of Mathematics\\ Georgia Institute of Technology \\ 686 Cherry St. \\ Atlanta, GA 30332}
\email{ctpartin@gatech.edu}

\begin{document}
\maketitle

\begin{abstract}
We derive a lower bound on the size of finite non-cyclic quotients of the braid group that is superexponential in the number of strands. We also derive a similar lower bound for nontrivial finite quotients of the commutator subgroup of the braid group.
\end{abstract}
\section{Introduction}
Let $\B_n$ denote the braid group on $n$ strands and let $\B_n'$ denote its commutator subgroup. It is a basic problem to describe all homomorphisms from $\B_n$ or $\B_n'$ to a given group $G$. The first results are due to Artin~\cite{artin}, who described all transitive homomorphisms from $\B_n$ to the symmetric group $S_n$. Lin~\cite{lin} extended Artin's results in various ways and proved analogous results for $\B_n'$. 

Since the abelianization of $\B_n$ is infinite cyclic, it is always possible to construct homomorphisms $\B_n\rightarrow G$ that factor through $\Z$. Such a homomorphism is said to be cyclic. On the other hand, the abelianization of $\B_n'$ trivial for $n\geq 5$ (see~\cite{lin}) and so no such construction is possible for $\B_n'$. In general, for fixed $n$ and $G$ it is often not clear whether there exist non-cyclic homomorphisms $\B_n\rightarrow G$ or non-trivial homomorphisms $\B_n'\rightarrow G$. Our main results are a necessary condition for the existence of non-cyclic homomorphisms $\B_n\rightarrow G$ and a necessary condition for the existence of non-trivial homomorphisms $\B_n'\rightarrow G$. 

\begin{theorem}\label{thm:main1}
Let $G$ be a finite group and let $n\geq 5$.  If $\B_n\rightarrow G$ is not a cyclic homomorphism then 
\[
|G| \geq 2^{\floor{n/2}-1}\left(\floor{n/2}\right)!
\]
\end{theorem}

\begin{theorem}\label{thm:main2}
Let $G$ be a finite group and let $n\geq 5$. If $\B_n'\rightarrow G$ is not the trivial homomorphism then 
\[
|G| \geq 2^{\floor{n/2}-2}\left(\floor{n/2}-1\right)!
\]
\end{theorem}

Perhaps the best known finite, non-cyclic  (respectively nontrivial) quotient of $\B_n$ (respectively $\B_n'$) is the symmetric group $S_n$ (respectively the alternating group $A_n$). It does not appear to be known whether there are any such quotients of $\B_n$ (respectively $\B_n'$) of smaller cardinality for $n\geq 5$ (respectively $n\geq 6$), although larger finite non-cyclic quotients do exist. We remark here that $\B_n$, and hence $\B_n'$, is residually finite and so they both possess plenty of finite quotients.

\medskip

After learning about our Theorem~\ref{thm:main1}, Margalit asked the following.

\begin{question}\label{q1}

For $n\geq 5$, is $S_n$ the smallest finite, non-cyclic quotient of $\B_n$? 

\end{question}

We also have the following related question.

\begin{question}\label{q2}

For $n\geq 6$, is $A_n$ the smallest finite, non-trivial quotient of $\B_n'$?
\end{question}

The first question above has a positive answer for $n=2$ and $n=3$ as can be checked by examining the list of groups of order at most 6. The second question has a positive answer for $n=5$ because any quotient of $\B_5'$ is perfect (see~\cite[p.7]{lin}) and $A_5$ is the smallest non-trivial perfect group. The answer to the first question is negative for $n=4$ because there is an exceptional surjective homomorphism $S_4\rightarrow S_3$. The second question has a negative answer for $n\in \{3,4\}$ also for exceptional reasons: the group $\B_3'$ is a free group and so surjects onto $\Z/2$, while $\B_4'$ surjects onto the commutator subgroup of $S_4$ which in turn surjects onto the (cyclic) commutator subgroup of $S_3$. 

The analogue of the above questions for mapping class groups asks for the minimal non-trivial quotient of the genus $g$ mapping class group $\Mod(S_g)$.  Zimmerman~\cite{zimmermann} proved that the smallest non-trivial quotient of $\Mod(S_g)$ is the symplectic group $\Sp_g(\Z/2)$ provided $g\in \{3,4\}$, and conjectured that the result held for all $g\geq 3$. Zimmermann's conjecture was later proven by Kielak--Pierro~\cite{kielakpierro}. In a slightly different direction, Berrick--Gebhardt--Paris~\cite{berrickgebhardtparis} proved that the minimal index of a proper subgroup of $\Mod(S_g)$ is equal to $2^{g-1}(2^g-1)$ and that, up to conjugation, there is exactly one subgroup with this index. We point out, though, that the subgroups of this index are not normal. 

\p{Obstructions} Various conditions are  known to obstruct the existence of non-cyclic homomorphisms $\B_n\rightarrow G$ or non-trivial homomorphisms $\B'_n\rightarrow G$. For example, it follows from the fact that $\B_n'$ is perfect for $n\geq 5$ (see~\cite[p.7]{lin}) that any homomorphism from $\B_n$ or $\B_n'$ to a solvable group is cyclic or trivial, respectively, for $n\geq 5$. More generally, $\B_n'$ does not admit non-trivial homomorphisms to residually solvable groups for $n\geq 5$. 

Another obstruction arises from sizes of generating sets. The braid group $\B_n$ admits a generating set of size 2 (consisting of a rotation and a half-twist), and hence any non-trivial quotient of $\B_n$ (in particular any non-cyclic quotient of $\B_n$) can also be generated by two elements. Likewise, Gorin--Lin~\cite[p.6]{lin} proved that $\B_3'$ has a generating set of size 2, and the second author proved~\cite{kordeksmallgen} that $\B_n'$ is generated by two elements for $n\in \{5\}\cup [7,\infty)$ and by three elements for $n\in \{4,6\}$. It follows that any quotient of $\B_n'$ is generated by two elements for $n\in \{5\}\cup [7,\infty)$ or by three element for $n\in \{4,6\}$. 

\p{Prior results} Since any finite group $G$ embeds into a sufficiently large symmetric group $S_k$, it is sometimes possible to understand all homomorphisms $\B_n, \B_n'\rightarrow G$ by classifying all homomorphisms $\B_n, \B_n' \rightarrow S_k$.  The first result of the latter type is due to Artin~\cite{artin}, who proved that all homomorphisms $\B_n\rightarrow S_n$ with transitive image and with $n\neq 4,6$ are either cyclic or conjugate to the standard projection. Artin also completely described all of the exceptional homomorphisms that arise for $n\in \{4,6\}$. 

These results were later greatly extended by Lin~\cite{lin}. Among other results, he proved that any homomorphism $\B_n\rightarrow S_k$ with $k<n$ and $n\geq 5$ is cyclic, that any transitive homomorphism $\B_n\rightarrow S_m$ with $6 <n<m <2n$ is cyclic, that all transitive homomorphisms $\B_n\rightarrow S_{n+1}$ with $n\geq 6$ are cyclic, and that all transitive homomorphisms $\B_n\rightarrow S_{n+2}$ with $n\geq 5$ are cyclic. Lin also completely characterized all homomorphisms $B_n\rightarrow S_{2n}$  with $n\geq 7$, and gave explicit formulas for the non-cyclic homomorphisms that arise. 

Lin also proved several results about homomorphisms from $\B_n'$ to symmetric groups. For example, he proved that any homomorphism $\B_n'\rightarrow S_n$ with $n\geq 5$ is the restriction of a homomorphism $\B_n\rightarrow S_n$, that any homomorphism $\B_n'\rightarrow S_k$ with $n\geq 5$ and $k < n$ is trivial, and that any transitive homomorphism $B_n'\rightarrow S_k$ with $k < 2n$ is primitive.

\subsection*{Overview}
In Section~\ref{sec:tss}, we first review the basic properties of totally symmetric subsets and give examples of totally symmetric subsets of braid groups. We then proceed to Proposition~\ref{prop:size} and its proof, which form the technical core of the paper. In Section~\ref{sec:proofs} we prove Theorem~\ref{thm:main1} and Theorem~\ref{thm:main2}.

\subsection*{Acknowledgments} The majority of this work was completed in the summer of 2019 at the Georgia Institute of Technology mathematics REU as a part of the research cluster on braids, which was led by Dan Margalit. The authors would like to thank him for his guidance, for many helpful conversations, and for his comments on early drafts of this paper. The authors also thank Dawid Kielak for pointing out an error in an earlier version of the paper. We are grateful to the referee for numerous helpful comments.  This material is based upon work supported by the National Science Foundation under Grant Nos. DMS - 1057874 and DMS - 1745583.

\section{Totally symmetric sets}\label{sec:tss}

 To prove Theorems~\ref{thm:main1} and~\ref{thm:main2} we will use the theory of totally symmetric sets, which were introduced by Margalit and the second author \cite{km}. A totally symmetric subset of a group $G$ is a finite subset $\{g_1,\ldots , g_n\}$ of $G$ such that
\begin{enumerate}
\item The elements $g_i$ pairwise commute, and
\item For any permutation $\sigma\in S_n$, there exists $h\in G$ such that 
\[
hg_ih^{-1} = g_{\sigma(i)}\quad
\text{for all}\quad 1\leq i\leq n.
\]
\end{enumerate}

The theory of totally symmetric sets is particularly powerful as a tool for analyzing group homomorphisms. This stems from the following fact: If $f: G\rightarrow H$ is a homomorphism and $S$ is a totally symmetric subset of $G$, then $f(S)$ is a totally symmetric subset of $H$.  

\p{Some examples of totally symmetric sets} Totally symmetric sets occur naturally in the study of braid groups and, more generally, mapping class groups of surfaces. We now describe two totally symmetric subsets $X\subset \B_n$ and $X'\subset \B_n'$ that will play critical roles in the proofs of Theorems~\ref{thm:main1} and~\ref{thm:main2}.

Recall that $\B_n$ is generated by a standard set of half-twists $\sigma_1,\ldots, \sigma_{n-1}$. The usual commutation relations along with the change-of-coordinates principle from mapping class group theory imply that the following subset of $\B_n$ is totally symmetric of size $\floor{n/2}$:
 \[ 
 X = \{\sigma_{2i-1}\}_{i=1}^{\floor{n/2}}.
\]
 
The elements of $\B_n'$ are exactly those elements of $\B_n$ with vanishing signed word length. We claim that the following subset of $\B_n'$ is totally symmetric of size $\floor{n/2}-1$:
\[
X' = \{\sigma_1\sigma_{2i-1}^{-1}\}_{i=2}^{\floor{n/2}}.
\]
It is clear that the elements of this set commute pairwise and that they can be permuted in an arbitrary fashion via conjugation by elements of $\B_n$. However, in order for $X'$ to be a totally symmetric subset of $\B_n'$ we require the conjugating elements to lie in $\B_n'$. This can be arranged as follows. Suppose that $g$ permutes the elements of $X'$ according to a permutation $\sigma\in S_{\floor{n/2}-1}$, and let $\ell$ denote the signed word length of $g$. Then $g\sigma_1^{-\ell}$ has vanishing signed word length and permutes the elements of $X'$ according to $\sigma$, since for each $1\leq i\leq \floor{n/2}$ we have
\[
(g\sigma_1^{-\ell})(\sigma_1\sigma_{2i-1}^{-1})(g\sigma_1^{-\ell})^{-1} = g(\sigma_1\sigma_{2i-1}^{-1})g^{-1}.
\]

\p{The image of a totally symmetric set} The following lemma establishes one of the most useful properties possessed by totally symmetric sets. It is a main ingredient in our proof of Theorems~\ref{thm:main1} and~\ref{thm:main2}. The proof is due to Margalit and the second author~\cite[Lemma 2.1]{km}, but we also give it here for the reader's convenience. 
\begin{lemma}\label{prop:TSSHoms}
 Let $G,H$ be groups and let $f:G\rightarrow H$ be a homomorphism. If $S\subset G$ is a totally symmetric set, then $|f(S)|$ is equal to either $1$ or $|S|$. 
 \end{lemma}
\begin{proof}

Let $S = \{g_1,\ldots, g_n\}$ and assume that $|f(S)| < |S|$.  After relabeling the elements of $S$ we may assume that $f(g_1) = f(g_2)$, and hence that $f(g_1g_2^{-1}) = 1$. Since $S$ is totally symmetric,  for each $i>2$ there exists $h\in G$ with $hg_1h^{-1} = g_1$ and $hg_2h^{-1} = g_i$. We have
\begin{align*}
f(g_1g_i^{-1}) = f(h(g_1g_2^{-1})h^{-1}) &= f(h)\left(f(g_1g_2^{-1})\right)f(h)^{-1} \\
& = f(h)\left(1\right)f(h)^{-1}\\
&= 1.
\end{align*}
That is, $f(g_i) = f(g_1)$. This shows that $|f(S)| = 1$, as desired. 
\end{proof}

\p{The lower bound} The remainder of this section is dedicated to the proof of Proposition~\ref{prop:size}, which gives a lower bound on the size of a group in terms of the size of a totally symmetric subset whose members have finite order. 
 \begin{proposition}\label{prop:size}
Let $n\geq 1$, and suppose that $S$ is a totally symmetric subset of a group $G$ with $|S| = n$. If the elements of $S$ have finite order, then $|G|\geq 2^{n-1}n!$.
\end{proposition}
We remark that Proposition~\ref{prop:size} is sharp in the sense that, for each $n\geq 1$, there exists a group $G$ of cardinality $2^{n-1}n!$ that contains a totally symmetric set of size $n$ all of whose elements have finite order. For $n=1$ this is trivial. For $n\geq 2$, we may take $G = S_n\ltimes V$, where 
\[
V = (\Z/2)^n/\langle e_1+e_2+\cdots +e_n\rangle
\]
is the standard representation of the symmetric group over $\Z/2$. It follows that $\{e_1,e_2,\ldots,e_n\}\subset G$ is a totally symmetric set of size $n$, and since $V$ is a $\Z/2$-vector space of dimension $n-1$ we have that $|G|=2^{n-1}n!$.

\medskip

To prove Proposition~\ref{prop:size}, we require the following lemma. It is due to Chen, Margalit, and the second author~\cite{chenkordekmargalit2}. Since their paper has not yet appeared, we give the proof here. 
\begin{lemma}
\label{lem:torsion}
Let $n\geq 1$, let $G$ be a group, and let $S \subseteq G$ be a totally symmetric subset with $|S| = n$.  Suppose that each element of $S$ has finite order.  Then $\langle S \rangle$ is a finite group whose order is greater than or equal to $2^{n-1}$.
\end{lemma}
\begin{proof}
If $n = 1$, then the lemma is trivially true, so we may assume that $n\geq 2$. Let $S = \{g_1,g_2,\ldots, g_n\}$. Since $S$ is totally symmetric, the elements $g_i$ all have the same order, which we denote by $m$. It follows that the subgroup $\langle S \rangle$ generated by $S$ is a quotient of $(\Z/m)^n$, which is finite. Thus $\langle S\rangle$ is finite. Our aim now is to show that $|\langle S\rangle| \geq 2^{n-1}$. 

Let $p$ be the smallest integer in the interval $[1,m]$ such that $g_i^p = g_j^p$ for each $i,j$ (note that such a $p$ exists because $g_i^m = g_j^m = 1$ for all $i,j$). Observe that in fact $p>1$, since $p = 1$ woud imply that the elements $g_i$ were not all distinct. 

Now let $A$ denote the free abelian group of rank $n$ with generators $e_1,\ldots, e_n$. Let $\bar A$ denote the quotient of $A$ by the subgroup of relations generated by the elements
\[
e_1 + \cdots + e_n \quad  \text{and}\quad  pe_i\ \ \text{with}\ 1 \leq i \leq n. 
\]
It follows that $|\bar A| = p^{n-1} \geq 2^{n-1}$. To prove the lemma, it suffices to demonstrate the existence of a surjective homomorphism $ \langle S\rangle \rightarrow \bar A$. We will do this by showing that the function $\phi: S \rightarrow \bar A$ defined by $\phi(g_i) = \bar e_i$ extends to a well-defined surjective homomorphism $\langle S\rangle \rightarrow \bar A$ (surjectivity follows from the fact that $\phi$ maps elements of $S$ to generators of $\bar A$). 

To this end, we consider the map $\phi': \Z S\rightarrow  A$ induced by $\phi$, where $\Z S$ denotes the free abelian group generated by the elements of $S$.  We will show that any relation satisfied by the elements of $S$ (that is, any element of the kernel of the quotient $\Z S\rightarrow \langle S\rangle$) maps to a defining relation for $\bar A$. We will then conclude that $\phi'$ descends to a well-defined surjective homomorphism $\langle S\rangle\rightarrow \bar A$, as desired. We first carry out these steps for $n\geq 3$ and then proceed to the case of $n=2$.

Let $R = g_1^{q_1}g_2^{q_2}\cdots g_n^{q_n}\in \Z S$ be a relation in $\langle S \rangle$ (here we write the elements of $\Z S$ in multiplicative notation). Since $g_1^p = g_i^p$ in $\langle S\rangle$ for each $i$, we may assume that $0\leq q_i\leq p-1$ for $2\leq i \leq n$. 

Assume that $n\geq 2$. First, we claim that the $q_i$ are independent of $i$ for $2\leq i \leq n$. The claim is trivially true if $n=2$, so from now on we assume that $n\geq 3$. By the total symmetry of $S$, there exists $h\in G$ that conjugates $g_2$ to $g_3$, that conjugates $g_3$ to $g_2$, and that commutes with all other $g_i$. This implies that the element $R' = g_1^{q_1}g_2^{q_3}g_3^{q_2}g_4^{q_4}\cdots g_n^{q_n}\in \Z S$ is also relation in $\langle S\rangle$. It follows now that 
\[
R(R')^{-1} = g_2^{q_2-q_3}g_3^{q_3-q_2}
\]
is a relation in $\langle S\rangle$ and hence that 
\[
g_2^{q_2-q_3} = g_3^{q_2-q_3}
\] 
in $\langle S\rangle$. After possibly relabeling $g_2$ and $g_3$ we may assume that $q_2 \geq q_3$, and since $0\leq q_2,q_3<p$, we also have $0\leq q_2-q_3 <p$. By total symmetry, we have that $g_i^{q_2-q_3} = g_j^{q_2-q_3}$ in $\langle S\rangle$ for all $i,j$. By the minimality of $p$, we have that $q_2-q_3 = 0$. Total symmetry further implies that $q_3 = q_4 = \cdots = q_n$. This proves the claim.

Next, we claim that there exists $\ell\in \Z$ such that $q_1 = q_2 + \ell p$. By the preceding claim, we can now write
\[
R = g_1^{q_1}(g_2g_3\cdots g_n)^{q_2}.
\]
As in the proof of the preceding claim, we may transpose $g_1$ and $g_2$ to obtain a further relation
\[
g_2^{q_1}(g_1g_3\cdots g_n)^{q_2}.
\]
Combining these two relations shows that  $g_1^{q_1-q_2}g_2^{-(q_1-q_2)}$ is also a relation, and hence that 
\[
g_1^{q_1-q_2} = g_2^{q_1-q_2}
\]
in $\langle S\rangle$. Since $g_i^p = g_j^p$ in $\langle S\rangle$ for all $i,j$, we further have that
\[
g_1^{q_1-q_2-\ell p} = g_2^{q_1-q_2-\ell p} \qquad\ \text{for all} \ \ell\geq 0. 
\]
The total symmetry of $S$ then implies that $g_i^{q_1-q_2-\ell p} = g_j^{q_1-q_2-\ell p}$ in $\langle S\rangle$ for all $i,j$. By the division algorithm, we may choose $\ell$ so that $0\leq q_1-q_2-\ell p\leq p-1$. Again by the minimality of $p$, we must have that $q_1-q_2-\ell p$. That is, $q_1 = q_2 + \ell p$. The claim is proven.

Combining the preceding two claims, we may now write
\[
R = g_1^{q_1}\cdots g_n^{q_n} =  g_1^{q_2+\ell p}(g_2\cdots g_n)^{q_2} = g_1^{\ell p}(g_1\cdots g_n)^{q_2}.
\]
It follows now that $R$ maps under $\phi'$ to the relation
\[
\ell(pe_1) + q_2(e_1+\dots + e_n)
\]
in $\bar A$. Since $R$ was chosen arbitrarily, this shows that $\phi'$ descends to a homomorphism $\langle S\rangle \rightarrow \bar A$. This completes the proof of the lemma.  
\end{proof}

We can now give the proof of Proposition~\ref{prop:size}.
\begin{proof}[Proof of Proposition~\ref{prop:size}]
We begin with the claim that if $G$ is a group and $S\subset G$ is totally symmetric of size $n$, then there is subgroup $\Gamma < G$ and a surjective homomorphism $\Gamma\rightarrow S_n$.  To prove the claim, we first observe that $G$ acts by conjugation on the set of totally symmetric subsets of $G$ of cardinality $n$. Explicitly, if $S = \{x_1,\ldots , x_n\}\subset G$ is totally symmetric then 
\[
g\cdot S = \{gx_1g^{-1},\ldots , gx_ng^{-1}\}.
\]
Let $\Gamma$ denote the stabilizer of $S$ in $G$. There is a homomorphism $\phi: \Gamma\rightarrow \text{Sym}(S)\cong S_n$ defined by sending $\gamma\in \Gamma$ to the automorphism defined by $s \rightarrow \gamma s \gamma^{-1}$.  Since $S$ is totally symmetric, every permutation of $S$ can be realized as the image of some element $\gamma\in \Gamma$. Thus $\phi$ is surjective and the claim follows. 

The kernel of $\phi$ consists of those elements of $\Gamma$ that commute with each element of  $S$.  Since the elements of $S$ commute pairwise, any element of $\langle S\rangle$ commutes with each element of $S$. This shows that $\langle S\rangle \subset \ker \phi$.  By Proposition~\ref{prop:size} we have that $|\langle S \rangle| \geq 2^{n-1}$, and so we further have that $|\ker \phi| \geq 2^{n-1}$. The proposition then follows from the fact that $|\ker\phi||S_n| = |\Gamma| \leq |G|$. 

\end{proof}

\section{The proofs of Theorems~\ref{thm:main1} and~\ref{thm:main2}}\label{sec:proofs}

We now have everything we need to prove Theorem~\ref{thm:main1} and Theorem~\ref{thm:main2}. 
\begin{proof}[Proof of Theorem~\ref{thm:main1}]
Let $f\colon \B_n\rightarrow G$ be any homomorphism that is not cyclic. We claim first that $|f(X)| = \floor{n/2}$. The theorem will follow from this claim, since Proposition~\ref{prop:size} will then imply that $|G|\geq 2^{\floor{n/2}-1}(\floor{n/2})!$. 

Since $|X| = \floor{n/2}$, we have that $|f(X)| \leq \floor{n/2}$.  For the sake of deriving a contradiction, assume that $|f(X)|\leq \floor{n/2}-1$. Lemma~\ref{prop:TSSHoms} then implies that $|f(X)| = 1$, and so we have that 
\[
f(\sigma_1) = f(\sigma_3).
\]
In other words, $\sigma_1\sigma_3^{-1}\in \ker f$. Since the normal closure of $\sigma_1\sigma_3^{-1}$ in $\B_n$ is equal to $\B_n'$ (see \cite{lin}) and $\ker f$ is a normal subgroup of $\B_n$, we have that $\B_n'\subset \ker f$. In view of the fact that $\B_n/\B_n' \cong \Z$, we have that $f$ factors through $\Z$. That is, $f$ is cyclic. This contradiction shows that $|f(X)| = \floor{n/2}$, as claimed. This completes the proof.
\end{proof}

We now give the proof of Theorem~\ref{thm:main2}.

\begin{proof}[Proof of Theorem~\ref{thm:main2}]
 We first dispense with the case $n=5$. As we pointed out in the introduction, any quotient of $\B_5'$ is perfect and the smallest non-trivial perfect group is the alternating group $A_5$. It follows that if $f: \B_5'\rightarrow G$ is non-trivial then $|G| \geq |A_5| = 60$. Since $60 > 2^{2-2}(2-1)! = 2^{\floor{n/2}-2}(\floor{n/2}-1)!$ the proof is complete in the case $n=5$. 

Assume now that $n\geq 6$. Let $f: \B_n'\rightarrow G$ be a non-trivial homomorphism. Recall from Section~\ref{sec:tss} that there is a totally symmetric subset $X'$ of $\B_n'$ of size $\floor{n/2}-1$ defined by 
\[
X' = \{\sigma_1\sigma_{2i-1}^{-1}\}_{i=2}^{\floor{n/2}}.
\]
Lemma~\ref{prop:TSSHoms} implies that either $|f(X')| = |X'| = \floor{n/2}-1$ or that $|f(X')| = 1$. Assume that $|f(X')| = \floor{n/2}-1$. It follows directly from Proposition~\ref{prop:size} that $|G| \geq 2^{\floor{n/2}-2}(\floor{n/2}-1)!$. To complete the proof, it therefore suffices to show that if $|f(X')| = 1$ then $f$ is trivial. 

We now assume that $|f(X')| = 1$. The assumption implies that  $f(\sigma_1\sigma_3^{-1}) = f(\sigma_1\sigma_5^{-1})$ and therefore that $f(\sigma_3\sigma_5^{-1}) = 1$. The element $\sigma_3\sigma_5^{-1}$ is conjugate in $\B_n$ to $\sigma_1\sigma_3^{-1}$, and, since the normal closure of $\sigma_1\sigma_3^{-1}$ in $\B_n$ is equal to $\B_n'$, the normal closure of $\sigma_3\sigma_5^{-1}$ in $\B_n$ is also equal to $\B_n'$. It follows that the normal closure of $\sigma_3\sigma_5^{-1}$ in $\B_n'$ is also equal to $\B_n'$ (see~\cite[Lemma 8.3]{chenkordekmargalit1}). It further follows that $\B_n' \subset \ker f$, and hence that $\B_n' = \ker f$. That is, $f$ is trivial. This completes the proof.
\end{proof}


\end{document}